\documentclass{amsart}

\newcommand{\cH}{\mathcal{H}}

\newcommand{\cS}{\mathcal{S}}

\newcommand{\cW}{\mathcal{W}}

\newcommand{\m}{\to} \newcommand{\Q}{\mathbb{Q}}\newcommand{\R}{\mathbb{R}}\newcommand{\C}{\mathbb{C}}
\newcommand{\p}{Poincar\'e }

\newcommand{\bN}{\mathbb{N}}

\newcommand{\bQ}{\mathbb{Q}}\newcommand{\bR}{\mathbb{R}}

\newcommand{\bZ}{\mathbb{Z}}

\usepackage[dvipsnames,svgnames,x11names,hyperref]{xcolor}
\usepackage{url,graphicx,verbatim,amssymb,enumerate,stmaryrd}
\usepackage[pagebackref,colorlinks,citecolor=Blue,linkcolor=Blue,urlcolor=Blue,filecolor=Blue]{hyperref}
\usepackage{tikz} 
\usetikzlibrary{arrows,decorations.pathmorphing,backgrounds,fit,positioning,shapes.symbols,chains,calc}
\usepackage{microtype,multicol}
\usepackage[all]{xy}
\usepackage[hmargin=3cm,vmargin=3cm]{geometry}

\newtheorem{theorem}{Theorem}[section]
\newtheorem{lemma}[theorem]{Lemma}
\newtheorem{proposition}[theorem]{Proposition}
\newtheorem{corollary}[theorem]{Corollary}
\newtheorem{conjecture}[theorem]{Conjecture}
\theoremstyle{definition}
\newtheorem{definition}[theorem]{Definition}
\newtheorem{example}[theorem]{Example}
\newtheorem{remark}[theorem]{Remark}

\newcommand{\mr}[1]{{\rm #1}}
\newcommand{\fS}{\mathfrak{S}}

\title{Homological stability for complements of closures}
\author{Alexander Kupers}
\thanks{Alexander Kupers is supported by a William R. Hewlett Stanford Graduate Fellowship, Department of Mathematics, Stanford University, and was partially supported by NSF grant DMS-1105058.}
\author{Jeremy Miller}
\date{\today}

\begin{document}

\begin{abstract}We prove Conjecture F from \cite{VW} which states that the complements of closures of certain strata of the symmetric power of a smooth irreducible complex variety exhibit rational homological stability. Moreover, we generalize this conjecture to the case of connected manifolds of dimension at least 2 and give an explicit homological stability range.\end{abstract}

\maketitle


\section{Introduction}

The goal of this paper is to prove a generalization of a conjecture of Vakil and Wood (Conjecture F of \cite{VW}). This conjecture concerns homological stability for certain subspaces of the symmetric powers defined as the complements of closures of certain strata. Here a sequence of spaces $X_k$ is said to have homological stability if the homology groups $H_i(X_k)$ are independent of $k$ for $k \gg i$.

We begin by defining the relevant subspaces of symmetric powers. Let $\mr{Sym}_k(M)$ denote the symmetric power $M^k/\fS_k$ of a space $M$. Here $\fS_k$ denotes the symmetric group on $k$ letters acting by permuting the terms. To any element of $\mr{Sym}_k(M)$ we can associate a way of writing the number $k$ as a sum of positive integers by counting with what multiplicity each point appears. Such a sum is called a \textit{partition} of $k$. For example, to the element $\{m_1,m_2,m_2,m_3\} \in \mr{Sym}_4(M)$ with $m_1 \neq m_2 \neq m_3$ we can associate the partition $4=1+1+2$. Using this one can define the following subspaces of $\mr{Sym}_k(M)$.

\begin{definition}\label{SDW} Let $\lambda$ be a partition of $k$. \begin{enumerate}[(i)] \item Let $S_\lambda(M)$ be the subspace of $\mr{Sym}_k(M)$ consisting of elements that have associated partition equal to $\lambda$. We call this the \textit{stratum corresponding to $\lambda$}.
\item Let $D_{\lambda}(M)$ be the closure of $S_{\lambda}(M)$ in $\mr{Sym}_k(M)$. This is the \textit{discriminant corresponding to $\lambda$}.
\item Let $W_\lambda(M)$ be the complement of $D_\lambda(M)$ in $\mr{Sym}_k(M)$. We call this the \textit{complement of the discriminant corresponding to $\lambda$}.
\end{enumerate}
\end{definition}

If $\lambda$ is a partition of $k$, we can obtain from it a partition of $k+1$ by adding another $1$ to $\lambda$. More generally we can add $j$ additional $1$'s to obtain a partition of $k+j$, which we denote by $1^j\,\lambda$. In other words, if $\lambda=m_1+\ldots +m_i$, then $1^j \lambda$ is the partition $1+ \ldots +1+m_1+\ldots + m_i$ where there are $j$ additional $1$'s. Some of these spaces are familiar. For example, the space $W_{1^{k-2}\,2}(M) = S_{1^{k}}(M)$ is the configuration space of $k$ distinct unordered points in $M$, often denoted $C_k(M)$. Similarly $W_{1^{k-c-1}\,c+1}(M)$ is called the bounded symmetric power of $M$, often denoted $\mr{Sym}^{\leq c}_k(M)$. It can be defined as the subspace of $\mr{Sym}_k(M)$ where no point of $M$ has multiplicity greater than $c$. For general $\lambda$, one can think of $W_\lambda(M)$ as those elements of $\mr{Sym}_k(M)$ that cannot be  made to have associated partition $\lambda$ by an arbitrarily small perturbation. We can now state Conjecture F of \cite{VW} .

\begin{conjecture}[Conjecture F]
For any irreducible smooth complex variety $X$, $\dim H_i(W_{1^j \lambda}(X);\bQ) = \dim H_i(W_{1^{j+1} \lambda}(X);\bQ)$ for $j \gg i$.
\end{conjecture}

We prove this conjecture, generalize it to the case of arbitrary connected manifolds of dimension at least 2 and give an explicit homological stability range. That is, we prove the following theorem.

\begin{theorem} \label{main} Let $M$ be a connected manifold of dimension $d\geq 2$, then we have that
\[H_i(W_{1^j \lambda}(M);\bQ) \cong H_i(W_{1^{j+1} \lambda}(M);\bQ)\]
for $i \leq j-1$ (except if $M$ is of dimension $2$ and non-orientable, in which it is $i \leq \frac{j}{2}-1$). 
\end{theorem}

We actually give a better range that depends on $M$ as well as $\lambda$, described by functions $f^\mr{or}_{M,\lambda}(j):\bN_0 \m \bN_0$ defined in Equation \ref{eqnfmlambda} in the orientable case and $f^\mr{nor}_{M,\lambda}(j):\bN_0 \m \bN_0$ defined in Equation \ref{eqnfmnonlambda} in the non-orientable case. Here $\bN_0$ denotes the non-negative integers. In particular, the slope of the function $f^\mr{or}_{M,\lambda}(j)$ is $a+1$ with $a<\dim M-1$ the largest integer such that the reduced rational homology groups $\tilde H_i(M;\bQ)$ vanish for $i \leq a$ (later referred to as condition $(*)_a$). The use of rational coefficients is essential in many parts of the argument but not all. See Remark \ref{remintegral} for a discussion of what results hold with integral coefficients.

In general, the isomorphism of Theorem \ref{main} is given by a transfer map which is described in Definition \ref{lemtransferinverse}. When $M$ is the interior of a manifold with non-empty boundary, one can also define a stabilization map $t: W_{1^j \lambda}(M) \m W_{1^{j+1} \lambda}(M)$ (see Definition \ref{defstab}) given by ``bringing a particle in from infinity.'' The stabilization map induces an isomorphism on rational homology in the same range as the transfer map. Our result uses homological stability for symmetric powers and configuration spaces of unordered distinct particles as input. It does not use homological stability for bounded symmetric powers and hence gives a new proof of Theorem 1.6 of \cite{kupersmillercompletions} with an improved range. Indeed, for orientable manifolds the range one obtains for $\mr{Sym}_k^{\leq c}(M)$ is $* \leq k-1$ if $\dim M>2$ and $* \leq \min(k-1,c-4+k)$ if $\dim M = 2$. 

Conjecture F was motivated by Theorem 1.30 of \cite{VW} which asserts that for $\lambda$ a partition of $k$, the limit of $[W_{1^j \lambda}(X)]/[\mr{Sym}_{j+k}(X)]$ as $j \to \infty$ exists in a certain localization of the dimensional completion of Grothendieck ring of complex varieties. As an abelian group, the Grothendieck ring of varieties is defined as the quotient of the free abelian group on the set of isomorphism classes of complex varieties, modulo the relation $[X] = [Z] + [X\backslash Z]$ whenever $Z \subset X$ is a closed subvariety. The ring structure is induced by Cartesian product. The conjecture is part of a larger question concerning the relationship between homological stability and stability in the Grothendieck ring of varieties. For many sequences of varieties with homological stability, Vakil and Wood were able to prove that the corresponding elements in the Grothendieck ring stabilize. For $W_{1^j \lambda}(X)$ the corresponding elements in the Grothendieck ring also stabilize, but homological stability was previously not known. Conjecture F is then obtained from the idea that one should expect homological stability in situations where there is stability in the Grothendieck ring and vice versa. 

There is a close but not exact relationship between the homology of a variety and its corresponding element in the Grothendieck ring. Using motivic zeta functions and a heuristic they dub ``Occam's razor for Hodge structures,'' Vakil and Wood also developed a procedure for predicting rational Betti numbers from elements of the Grothendieck ring. This perspective was designed to explain the apparent correlation between the two types of stability and give a prediction for the stable homology. Although there are some examples where Vakil and Wood's predictions of the stable homology is incorrect, see e.g. \cite{kupersmillernote} and \cite{tommasi}, we know of no examples where they make incorrect predictions regarding whether a sequence of spaces has rational homological stability. It would be interesting to know under what conditions stability in the Grothendieck ring is in fact equivalent to rational homological stability.

\begin{remark}We recently learned that TriThang Than independently proved our main theorem for connected open oriented manifolds $M$ of dimension greater than one that are the interior of a compact manifold with boundary \cite{trithang}. His work similarly uses compactly-supported cohomology, but filters $W_{1^j \lambda}(M)$ instead of its complement. The range he obtains is $* \leq \frac{j+k}{4}-\frac{1}{2}$.\end{remark}

\subsection{Outline} In Section \ref{secCS}, we describe a spectral sequence for computing compactly supported cohomology associated to an open filtration. In Section \ref{evenSec}, we prove Theorem \ref{main} when $M$ is an even dimensional orientable manifold which is the interior of a  manifold with non-empty boundary. In odd dimensions or when $M$ is not orientable, there are extra complications stemming from the fact that $W_{1^{j+1} \lambda}(M)$ is not orientable. In Section \ref{oddSec} we discuss how to modify the proof when $\dim M$ is odd but $M$ is orientable and in Section \ref{secnonorientable} we address the case that $M$ is not orientable. In Section \ref{secpuncturing}, we discuss how to remove the hypothesis that $M$ is the interior of a manifold with non-empty boundary.  

\subsection{Acknowledgements} We would like to thank Martin Bendersky, Tom Church, S{\o}ren Galatius, Sam Nariman, Arnav Tripathy, Ravi Vakil, and Melanie Wood for  helpful discussions.

\section{Compactly supported cohomology}\label{secCS}

The use of compactly supported cohomology for proving homological stability results was pioneered by Arnol'd in \cite{Ar}. In this section we review basic properties of compactly supported cohomology and describe a spectral sequence associated to an open filtration. If $N_j$ is a sequence of orientable manifolds each of dimension $n_j$, then $H_*(N_i) \cong H_*(N_{i+1})$ for $* \leq r_i$ if and only if $H^{*}_c(N_i) \cong H^{*+n_{i+1}-n_i}_c(N_{i+1})$ for $* \geq n_i-r_i$. Thus, for manifolds homological stability is equivalent to stability with a shift for compactly supported cohomology. In some situations, compactly supported cohomology is easier to use than homology because of the following fact. Suppose $X_i=U_i \cup C_i$ with $C_i$ closed and $U_i \cap C_i = \emptyset$. If the sequences $U_i$ and $C_i$ each exhibit homological stability, it does not necessarily follow that the homology of the sequence $X_i$ stabilizes. However, if $U_i$ and $C_i$ have stability for compactly supported cohomology, one can often deduce that $X_i$ has stability for compactly supported cohomology using following exact sequence (for example see III.7.6 of \cite{iversen}).

\begin{proposition} \label{exactseq} Let $X$ be a locally compact Hausdorff space and $C \subset X$ a closed subspace with open complement $U = X \backslash C$. There is a long exact sequence in compactly supported cohomology
\[\ldots \to H^*_c(U) \to H^*_c(X) \to H^*_c(C) \to H^{*+1}_c(U) \to \ldots\]
The same holds for compactly supported cohomology with coefficients in a local coefficient system on $X$.
\end{proposition}

Iterating this gives the following spectral sequence associated to an open filtration. We are unaware of a reference so we sketch a proof.

\begin{proposition}\label{specOpen} Let $X$ be a locally compact Hausdorff space and 
\[\ldots = U_{M-1} = U_{M} = \emptyset \subset U_{M+1} \subset \ldots \subset X = U_N =  U_{N+1} = \ldots\]
be an increasing sequence of open subsets of $X$. Then there is a spectral sequence  converging to $H^{p+q}_c(X)$ with $E^1$-page given by
\[E^1_{p,q} = H^{p+q}_c(U_{p}\backslash U_{p-1})\]
There is a similar spectral sequence for compactly supported cohomology with coefficients in any local coefficient system on $X$. It is natural with respect to open embeddings compatible with the filtrations.
\end{proposition}

\begin{proof}The idea is to splice together for each $i$ the long exact sequences for the inclusions of a closed set and its complement $U_i \backslash U_{i-1} \hookrightarrow U_i  \hookleftarrow U_{i-1}$, given by:
\[\cdots \to H^*_c(U_{i-1}) \to H^*_c(U_i) \to H^*_c(U_i \backslash U_{i-1}) \to H^{*+1}_c(U_{i-1}) \to \cdots\]

We can invoke Proposition \ref{exactseq} in this situation because open subsets of a locally compact Hausdorff space are again locally compact Hausdorff spaces. Next consider the following exact couple

\noindent\begin{minipage}{.5\linewidth}
\[\xymatrix{A \ar[rrrr]^i_{(1,-1)} & & & & A \ar[dll]^j_{(0,0)} \\
   & & E \ar[ull]^k_{(-1,0)} & & }\]
\end{minipage}%
\begin{minipage}{.3\linewidth}
  \begin{align*}A_{p,q} &= H^{p+q}_c(U_p)\\
  E_{p,q} &= H^{p+q}_c(U_p \backslash U_{p-1})\end{align*}
\end{minipage} 

\noindent with $(a,b)$ denoting the shift in bigrading. Here $i$ is the sum of the maps $H^{*}_c(U_{p-1}) \to H^{*}_c(U_p)$ induced by the inclusion of open subsets, $j$ is induced by the restriction map $H^*_c(U_p) \to H^*_c(U_p \backslash U_{p-1})$ along the closed inclusion $ U_p \backslash U_{p+1} \hookrightarrow U_p$ and $k$ is given by the boundary maps $H^*_c(U_p \backslash U_{p-1}) \to H^{*+1}_c(U_{p-1})$. 

The general machinery of exact couples gives us a spectral sequence with 
\[E^1_{p,q} = H^{p+q}_c(U_p \backslash U_{p-1}) \]
which will have homological differentials, i.e. $d_r$ is of bidegree $(-r,r-1)$.

To check that this converges and compute what it converges to, it suffices to note that a spectral sequence of an exact couple has $E^\infty$-page equal to the associated graded of a filtration $F^s A_{\infty,*}$ of $A_{\infty,*}$ if there are only finite many $p$ such that $i: A_{p,q} \to A_{p+1,q-1}$ is not an isomorphism and $A_{-\infty,*} = 0$. In our case $A_{\infty,*} =  H^*_c(X)$ and the condition for convergence holds, because $A_{p,q}$ is $0$ for $p$ sufficiently small and $A_{p,q} = H^{p+q}_c(X)$ for $p$ sufficiently large.
\end{proof}

\section{The proof for open oriented manifolds of even dimension}\label{evenSec}

In this section we assume that we are working with a connected oriented manifold $M$ of even dimension $d = 2n$ that is the interior of a manifold with non-empty boundary. We also assume that $d$ is at least $2$.

Our strategy in this section can be summarized as follows. We will filter $D_{\lambda}(M)$ by open sets whose differences are spaces where stability for compactly supported cohomology is already known. We will use the spectral sequence from Proposition \ref{specOpen} to show that $D_{1^j\lambda}(M)$ has stability for compactly supported cohomology. Using homological stability for symmetric powers and \p duality, we will see that symmetric powers have stability for compactly supported cohomology. Since $\mr{Sym}_{k+j}(M)=D_{1^j\lambda}(M) \cup W_{1^j\lambda}(M)$, we will be able to deduce stability for the compactly supported cohomology of $W_{1^j\lambda}(M)$ using Proposition \ref{exactseq}. Using Poincar\'e duality, this is equivalent to rational homological stability for the spaces $W_{1^j\lambda}(M)$. 

We start by defining the stabilization map. Let $M$ be the interior of $\bar{M}$, a connected manifold with non-empty boundary $\partial \bar{M}$. We do not require $\bar{M}$ to be compact. Let $D^k$ denote the open $k$-disk and $\bar{D}^k$ denote the closed $k$-disk. Pick an embedding $\phi: \bar{D}^{2n-1} \hookrightarrow \partial \bar{M}$ and homeomorphism \[\psi: \mr{int}(\bar{M} \cup_{\phi} \bar{D}^{2n-1} \times [0,1)) \to M\] whose inverse is isotopic to the inclusion of $M$ into  $\mr{int}(\bar{M} \cup_{\phi} \bar{D}^{2n-1} \times [0,1))$. 

\begin{definition} \label{defstab}
The stabilization map $t: \bR^{2n} \times W_{1^j \lambda}(M) \to W_{1^{j+1} \lambda}(M)$ is defined as follows: For $z \in \bR^{2n}$ and $x \in W_{1^j \lambda}(M)$, let $x \cup z$ be the element of $W_{1^{j+1} \lambda}(\mr{int}(\bar{M} \cup_{\phi} \bar{D}^{2n-1} \times [0,1)))$ given by $x$ in $M$ and $z \in \bR^{2n} \cong D^{2n-1} \times (0,1)$. Define $t$ by the formula $t(z,x) = \hat \psi(x \cup z)$ where \[\hat \psi:W_{1^{j+1} \lambda}(\mr{int}(\bar{M} \cup_{\phi} \bar{D}^{2n-1} \times [0,1))) \to W_{1^{j+1} \lambda}(M)\] is the map induced by applying $\psi$ to every point in the configuration.
\end{definition}

Note that this map depends on a choice of embedding and homeomorphism. However, up to homotopy, it only depends on choice of component of $\partial \bar{M}$. Since $\bR^{2n}$ is contractible, $H_*(\bR^{2n} \times W_{1^j \lambda}(M))=H_*( W_{1^j \lambda}(M)) $.  However, our proof of homological stability will use compactly supported cohomology. From this perspective, the copy of $\bR^{2n}$ is relevant. In particular, it makes the stabilization map an open embedding so it induces a map on compactly supported cohomology. In a similar fashion, one can define stabilization maps for the discriminant $D_{\lambda}(M)$, stratum $S_{\lambda}(M)$ and symmetric power $\mr{Sym}_k(M)$. To state the main result of this section, we need the following definition.

\begin{definition}\label{defcondasta} A manifold $M$ is said to satisfy condition $(*)_a$ for $a<\dim M-1$ if we have that $\tilde H_i(M;\bQ) = 0$ for $ i \leq a$.\end{definition} 

The goal of this section is to prove the following proposition. 

\begin{proposition}\label{propevenopen} Let $M$ be a connected oriented manifold of even dimension $d = 2n$ that is the interior of a manifold with boundary. The stabilization map $t_*: H_i(W_{1^j \lambda}(M);\bQ) \to H_i(W_{1^{j+1} \lambda}(M);\bQ)$ induces an isomorphism for $i \leq f^\mr{or}_{M,\lambda}(j)$, with \begin{equation}\label{eqnfmlambda} f^\mr{or}_{M,\lambda}(j) = \begin{cases} \min(k+j,d(k-r)+j-1)-1 & \text{if $\dim M = d >2$ and $H_1(M;\bQ) \neq 0$} \\
\min(k+j,2(k-r)+j-2)-1 & \text{if $\dim M = 2$ and $H_1(M;\bQ) \neq 0$} \\
\min((a+1)(k+j),d(k-r) +(a+1)j-2)-1 & \text{if condition $(*)_a$ holds for $a \geq 1$}\end{cases}
\end{equation}
\end{proposition}

In the next two sections we will see that the same is true for odd dimensional manifolds and manifolds which are not orientable. However, in those cases there will be some orientation issues. These do not occur if the dimension is even and the manifold is orientable, the case to which the original conjecture pertains, so for the convenience of the reader we separated the proofs.

An \textit{elementary collapse} of a partition $\lambda$ is a partition $\lambda'$ which is identical to $\lambda$ except two integers have been replaced by their sum. A partition $\lambda$ is called a \textit{collapse} of $\lambda'$ if $\lambda'$ can be constructed from $\lambda$ by a sequence of elementary collapses. For example, $1+2+2+4$ is a collapse of $1+1+1+1+2+3$. For $\lambda$ a partition of $k$, let $\mr{col}(\lambda)$ be the set of partitions of $k$ $\lambda'$ such that $\lambda$ is a collapse of $\lambda'$. If $\lambda$ consists of $r$ integers and $\lambda'$ consists of $p$ integers, then the \textit{depth} of $\lambda'$ is defined to be $r-p$. It is also equal to the number of elementary collapses necessary to obtain $\lambda'$ from $\lambda$. Let $\mr{col}_i(\lambda)$ be the subset of $\mr{col}(\lambda)$ of collapses of depth $i$.

\begin{example}If $\lambda = 1+2$, then $\mr{col}_0(\lambda) = \{\lambda\}$, $\mr{col}_1(\lambda) = \{3\}$ and $\mr{col}_i(\lambda) = \emptyset$ for $i \geq 2$.\end{example}

These collapses correspond to the strata in a stratification of $D_{\lambda}(M)$; the partition associated to any $x \in D_{\lambda}(M)$ must be a collapse of $\lambda$ and the subspace of elements corresponding to collapses of depth $i$ is a union of codimension $2ni$ submanifolds. The discriminant $D_{\lambda}(M)$  is the union of the spaces $S_{\lambda'}(M)$ over partitions $\lambda' \in \mr{col}(\lambda)$. Likewise, one can define $W_{\lambda}(M)$ as the union of $S_{\lambda'}(M)$ over partitions $\lambda' \notin \mr{col}(\lambda)$. To apply the spectral sequence from Proposition \ref{specOpen} to study the compactly supported cohomology of $ D_{\lambda}(M)$, we need to define a filtration of $ D_{\lambda}(M)$. To that end, we  introduce the following notation:
\[\cS_{i} = \bigcup_{\lambda' \in \mr{col}_{i}(1^j \lambda)}S_{\lambda'}(M) \qquad \qquad U_{i} = \begin{cases} \emptyset & \text{if $i < 0$} \\
D_{\lambda}(M) \backslash \bigcup_{i' \geq i+1} \cS_{i'} & \text{if $i \geq 0$}\end{cases}\]
Note that $U_{i-1} \subset U_{i}$ is either an equality or the inclusion of an open subspace with complement $\cS_{i}$, that $U_i = \emptyset$ for $i < 0$ and $U_i = {D}_{\lambda}(M)$ for $i \gg 0$. Using this filtration or its product with $\bR^{2n}$ Proposition \ref{specOpen} gives the following corollary.

\begin{corollary}\label{specEven} There exists a spectral sequence converging to $H^{p+q}_c(D_{\lambda}(M);\bQ)$ with $E^1$-page 
\[E^1_{p,q} = \begin{cases}\bigoplus_{\lambda' \in \mr{col}_{p}(\lambda)} H^{p+q}_c(S_{\lambda'}(M);\bQ) & \text{if $p\geq 0$} \\
0 & \text{if $p < 0$}\end{cases}\] 
There is a similar spectral sequence converging to $H^{p+q}_c(\bR^{2n} \times D_{\lambda}(M);\bQ)$ with $E^1$-page 
\[E^1_{p,q} = \begin{cases}\bigoplus_{\lambda' \in \mr{col}_{p}(\lambda)} H^{p+q}_c(\bR^{2n} \times S_{\lambda'}(M);\bQ) & \text{if $p\geq 0$} \\
0 & \text{if $p < 0$}\end{cases}\] 
\end{corollary}

\begin{lemma}\label{lemdisknat} The stabilization map $t_*:H_c^*(\R^{2n} \times D_{\lambda}(M)) \m H_c^*(D_{1\,\lambda}(M))$ respects the spectral sequences of Corollary \ref{specEven}. Moreover, the map on the $E^1$-pages is induced by the stabilization maps \[t_*:H^*_c(\R^{2n} \times S_{\lambda'}(M)) \m H^*_c(S_{1\,\lambda'}(M))\] \end{lemma}

\begin{proof}Recall that open embeddings induce maps on compactly supported cohomology via extension by zero. The stabilization map is an open embedding compatible with the filtrations used in Corollary \ref{specEven}. 
\end{proof}

The stabilization map $t$ induces a map $t: \mr{col}_p(1^j \lambda) \to \mr{col}_p(1^{j+1} \lambda)$ by $\lambda' \mapsto 1\,\lambda'$.

\begin{lemma}\label{lemdiskcolumnindex} \label{ones} For fixed $p$ the map $t: \mr{col}_p(1^j \lambda) \to \mr{col}_p(1^{j+1} \lambda)$ is an isomorphism if $j > p$. \end{lemma}

Given a partition $\lambda$, the number of $1$'s in $\lambda$ is defined to be the largest number $i$ such that $\lambda=1^i \lambda'$ for some partition $\lambda'$. 

\begin{lemma}
Each $\lambda' \in \mr{col}_p(1^j \lambda)$ has at least $j-p$ $1$'s.
\end{lemma}

To prove compactly supported cohomological stability for $D_{1^j \lambda}(M)$ using Corollary \ref{specEven}, we need to know that the spaces $S_{1^j \lambda'}(M)$ have compactly supported cohomological stability. Since these spaces are manifolds, compactly supported cohomological stability is equivalent to homological stability. 

\begin{lemma}\label{lemopenstratastab} Let $\lambda$ be a partition with $i$ $1$'s. The map $t_*: H_*(S_{\lambda}(M);\bQ) \to H_*(S_{1\,\lambda}(M);\bQ)$ is an isomorphism in the following ranges: (i) $* \leq i$ if $M$ is of dimension $d>2$, (ii) $* < i$ if $M$ is of dimension $d=2$, and (iii) $* < (a+1)i$ if condition $(*)_a$ holds.\end{lemma}

\begin{proof}

Let $\lambda'$ be the partition such that $1^i \lambda' =\lambda$ and let $r$ be the cardinality of $\lambda'$, i.e. $\lambda'$ is given by $m_1 + \ldots + m_r$. There is a fiber bundle: \[S_{1^{i}}(M \backslash \{r \text{ points}\}) \to S_{\lambda}(M) \to S_{\lambda'}(M)\] The map $S_{ \lambda}(M) \to S_{\lambda'}(M)$ is the map that forgets all of the points labeled by the number 1. The stabilization map induces a map from the Serre spectral sequence for this fibration to the Serre spectral sequence for fibration associated to $S_{1\, \lambda}(M)$.

The result now follows by spectral sequence comparison and the homological stability ranges of Church and Randal-Williams: a range $*<i$ for all dimensions $\geq 2$ from Corollary 3 of \cite{Ch}, a range $*\leq i$ for all dimensions $\geq 3$ from Theorem B of \cite{RW} and the improved range with vanishing reduced Betti numbers from Proposition 4.1 of \cite{Ch}. 

Two remarks are in order. Firstly, Church's results concern the transfer map, not the stabilization map. This is not an issue as Lemma \ref{lemtransferinverse} shows that these maps are rationally mutually inverse in the stable range. Secondly, if $M$ has the property that $\tilde H_i(M;\bQ) = 0$ for $ i \leq a$ with $a < \dim M -1$, then by Mayer-Vietoris $M \backslash \{r \text{ points}\}$ has the same property.\end{proof}

\begin{corollary}\label{coropenstratastab} Let $\lambda$ be a partition of $k$ with cardinality $r$. For each $\lambda' \in \mr{col}_p(1^j \lambda)$, the map $t_*: H^*_c(\R^{2n} \times S_{\lambda'}(M);\bQ) \to H^*_c(S_{1\,\lambda'}(M);\bQ)$ is an isomorphism in the following ranges: (i) $* \geq 2n(j+r-p+1)-j+p$ if $\dim M >2$, (ii) $* > 2n(j+r-p+1)-j+p$ if $\dim M =2$, and (iii) $* > 2n(j+r-p+1)- (a+1)(j+p)$ if condition $(*)_a$ holds.\end{corollary}

\begin{proof}
Pick $\lambda' \in \mr{col}_p(1^j \lambda)$ and let $i$ be the number of $1$'s in $\lambda'$. Note that $\R^{2n} \times S_{\lambda'}(M)$ and $S_{1\,\lambda'}(M)$ are orientable manifolds of dimension $2n(r+j-p+1)$. Using \p duality and Lemma \ref{lemopenstratastab}, we see that $t_*: H^*_c(\R^{2n} \times S_{\lambda'}(M);\bQ) \to H^*_c(S_{\lambda'}(M);\bQ)$ is an isomorphism in the ranges: (i) $* \geq 2n(j+r-p+1)-i$ if $\dim M >2$, (ii) $* > 2n(j+r-p+1)-i$ if $\dim M =2$, and (iii) $* > 2n(j+r-p+1)-(a+1)i$ if condition $(*)_a$ holds. By Lemma \ref{ones}, $i$ is at least $j-p$. The claim follows.\end{proof}

Lemma \ref{lemdiskcolumnindex} and Corollary \ref{coropenstratastab} imply that the stabilization map induces isomorphisms between large portions of the $E^1$-pages of the spectral sequences associated to $\R^{2n} \times D_{1^j \lambda}(M)$ and $D_{1^{j+1} \lambda}(M)$.

\begin{corollary}\label{Deven} 
Let $\lambda$ be a partition of $k$ with cardinality $r$. The stabilization map induces an isomorphism \[t_*: H^*_c(\bR^{2n} \times D_{1^j \lambda}(M);\Q) \m H_c^*( D_{1^{j+1} \lambda}(M) ;\Q)\] in the following ranges: (i) $* \geq 2n(r+j+1) -j+1$ if $\dim M >2$, (ii) $*> 2n(r+j+1)-j+1$ if  $\dim M =2$, and (iii) $* > 2n(j+r+1)-(a+1)j+1$ if condition $(*)_a$ holds.
\end{corollary}

\begin{proof} In Lemma \ref{lemdisknat}, we observed that the stabilization map induces maps between the two spectral sequences of Corollary \ref{specEven}. The $E^1$-page of the spectral sequence associated to the open filtration of $\R^{2n} \times D_{1^j \lambda}(M)$  has columns given by $0$ if $p < 0$ and by the direct sum $\bigoplus_{\lambda' \in \mr{col}_{p}(1^j \lambda)} H^{p+q}_c(\bR^{2n} \times S_{\lambda'}(M);\bQ)$ for $p\geq 0$. The $E^1$-page of the spectral sequence associated to the open filtration of $D_{1^{j+1} \lambda}(M)$ has columns given by $0$ if $p < 0$ and by the direct sum $\bigoplus_{\lambda' \in \mr{col}_{p}(1^{j+1} \lambda)} H^{p+q}_c(\bR^{2n} \times S_{\lambda'}(M);\bQ)$ for $p \geq 0$. 

We can combine these into a relative spectral sequence converging to the relative compactly supported cohomology of the stabilization map on the discriminants with $E^1$-page relative compactly supported cohomology of the stabilization maps on the strata. We make the following remarks about the $E^1$-page of the relative spectral sequence.
\begin{enumerate}[(a)]
\item The $E^1$-page is zero for $p<0$.
\item For dimension reasons, if $p \geq 0$ the $p$th column is concentrated in degrees $-p \leq q \leq 2n(r+j-p+1)-p$.
\item The stabilization map induced an isomorphism on the indexing set of the direct sum for $p < j$ by Lemma \ref{lemdiskcolumnindex} and for $0 \leq p < j$ by Corollary \ref{coropenstratastab} the stabilization map induced an isomorphism on each summand in $p$th column in the ranges (i) $q \geq 2n(j+r-p+1)-j+p-p = 2n(j+r-p+1)-j$ if $\dim M > 2$, (ii) $q > 2n(j+r-p+1)-j$ if $\dim M = 2$, and (iii) $q> 2n(j+r-p+1)-(a+1)(j-p)-p$ if the first $a$ reduced Betti numbers vanish. Thus for $0 \leq p < j$ the $E^1$-page of the relative spectral sequence vanishes in the same range.\end{enumerate}

From this we can find vanishing ranges for the $E^1$-page of the relative spectral sequence:
\begin{enumerate}[(i)]
\item If $\dim M > 2$, then the spectral sequence vanishes in the range $p+q \geq 2n(r+j+1)-j$. We check this by noting that for $0 \leq p < j$ we have that $2n(r+j+1)-j-p \geq 2n(j+r-p+1)-j$ and that for $p \geq j$ we have that $2n(r+j+1)-j-p \geq 2n(r+j-p+1)-p$.
\item If $\dim M = 2$, then the spectral sequence vanishes in the range $p+q > 2n(r+j+1)-j$ by the same arguments.
\item If the first $a$ reduced Betti numbers vanish, then the spectral sequence vanishes in the range $p+q > 2n(r+j+1)-(a+1)j$. We check this by noting that for $0 \leq p < j$ we have that $2n(r+j+1)-(a+1)j-p > 2n(j+r-p+1)-(a+1)(j-p)-p$. and that for $p \geq j$ we have that $2n(r+j+1)-(a+1)j-p \geq 2n(r+j-p+1)-p$ since their difference is $2np-(a+1)j$ and we have $a < 2n-1$ and $p \geq j$.
\end{enumerate}

Finally we need to convert a statement about relative homology to one about isomorphisms in a range, decreasing the range by 1.
\end{proof}

The next ingredient is rational stability for compactly supported cohomology of symmetric powers.

\begin{lemma}\label{lemmaopensymstab} The stabilization map for symmetric powers induces an isomorphism \[t_*:  H^*_c(\R^{2n} \times \mr{Sym}_{k+j}(M);\Q) \to H^*_c(\mr{Sym}_{k+j}(M);\Q)\] for $* \geq 2n(k+j+1)  -(k+j)$. If condition $(*)_a$ holds the range can be improved to $* \geq 2n(k+j+1) - (a+1)i$.\end{lemma}

\begin{proof}

Steenrod proved that $t$ induces an isomorphism on homology in the range $*\leq k+j$ (see e.g. Equation 22.7 of \cite{Srod}). The improved range $* \leq (a+1)(k+j)$ if condition $(*)_a$ holds can be deduced from the explicit description of the homology given by Milgram in \cite{Mil}. The lemma would follow from \p duality if the symmetric powers of $M$ were orientable manifolds. They are not manifolds in dimensions greater than two, but they are rational \p duality spaces. To see this, first note that the space $M^i$ is a manifold and is orientable if $M$ is. Since $M$ is even dimensional, the symmetric group action respects the orientation. Thus the \p duality isomorphism $H_*(M^i) \cong H^*_c(M^i)$ is $\fS_i$-equivariant. If $X$ is a locally finite CW-complex and $G$ is a finite group, then $H_*(X;\Q)_G \cong H_*(X/G;\Q)$. Similarly under these conditions we have $H^*_c(X;\Q)^G \cong H^*_c(X/G;\Q)$. Here the subscript on $H_*(X;\Q)_G$ denotes the coinvariants and the superscript $H^*_c(X;\Q)^G$ denotes the invariants. Since Poincar\'e duality interchanges coinvariants and invariants, $H_*(M^i/\fS_i;\Q) \cong H^*_c(M^i/\fS_i;\Q)$. Thus symmetric powers have rational \p duality and so the claim follows.\end{proof}

\begin{remark}\label{remsymstab} Actually, one can avoid the input of stability in compactly supported cohomology for symmetric powers. Since $\mr{Sym}_{j}(M) = D_{1^{j}}(M)$, Corollary \ref{Deven} applied to $\lambda = \emptyset$ shows that the compactly supported cohomology of $\mr{Sym}_{j}(M)$ stabilizes. Using \p duality, this gives a new proof of rational homological stability for symmetric powers. Unfortunately, this proof gives ranges that are worse by a shift of 1 if $\dim M >2$, or 2 if $\dim M = 2$ or when condition $(*)_a$ holds.\end{remark}

We now prove Proposition \ref{propevenopen}

\begin{proof}[Proof of Proposition \ref{propevenopen}] Note that $D_{1^j \lambda}(M)$ is a closed subspace of $\mr{Sym}_{k+j}(M)$ with complement $W_{1^j \lambda}(M)$. Proposition \ref{exactseq} gives two long exact sequences and the stabilization maps gives maps between the long sequences making the following diagram compute: 
\[\xymatrix{\ldots \ar[r] & H^*_c(\R^{2n} \times W_{1^j \lambda}(M)) \ar[d] \ar[r] & H^*_c(\R^{2n} \times \mr{Sym}_{k+j}(M);\Q) \ar[r] \ar[d] & H^*_c(\R^{2n} \times D_{1^j \lambda}(M);\Q) \ar[r] \ar[d] & \ldots \\ \ldots \ar[r] & H^*_c( W_{1^{j+1} \lambda}(M)) \ar[r] & H^*_c(\mr{Sym}_{k+j+1}(M);\Q) \ar[r] & H^*_c(D_{1^{j+1} \lambda}(M);\Q) \ar[r] & \ldots}\] Using Corollary \ref{Deven}, Lemma \ref{lemmaopensymstab} and the five lemma, we see that the stabilization map induces an isomorphism $t_*: H^*_c(\bR^{2n} \times W_{1^j \lambda}(M);\Q) \m H_c^*( W_{1^{j+1} \lambda}(M) ;\Q)$ in the following ranges: (i) $* \geq \max(2n(k+j+1)-(k+j),2n(r+j+1)-j+1)+1$ if $\dim M >2$, (ii) for $* \geq \max(2n(k+j+1)-(k+j),2n(r+j+1) -j+2)+1$ if $\dim M =2$, and (iii) if the first $a$ reduced Betti numbers vanish for $* \geq \max(2n(k+j+1)-(a+1)(k+j),2n(r+j+1) -(a+1)j+2)+1$. The shifts by 1 come from the application of the five lemma.

By an argument similar to that used in Lemma \ref{lemmaopensymstab}, the spaces $\bR^{2n} \times W_{1^j \lambda}(M)$ and $W_{1^{j+1}\lambda}(M)$ have rational \p duality and are of dimension $2n(k+j+1)$. This translates stability for compactly supported cohomology to homological stability with ranges (i) $* \leq \min(k+j,2n(k-r)+j-1)-1$ if $\dim M > 2$, (ii) $* \leq \min(k+j,2n(k-r)+j-2)-1$ if $\dim M = 2$, and (iii) $* \leq \min((a+1)(k+j),2n(k-r) +(a+1)j-2)-1$ if condition $(*)_a$ holds. This is the homological stability range claimed in Equation \ref{eqnfmlambda} of Proposition \ref{propevenopen}.

\end{proof}

\begin{remark}\label{remintegral}
If $\dim M >2$, we are forced to work with rational coefficients because the spaces $W_{1^j \lambda}(M)$ do not have integral \p duality. However, if $\dim M=2$ we can prove stability integrally, because $W_{1^j \lambda}(M)$ is a manifold and hence has integral \p duality. Almost all of the arguments of this section go through unchanged to prove that $H_*(W_{1^j \lambda}(M);\bZ)$ stabilizes if $\dim M=2$. The one modification needed is the following. Instead of using the results of Church in \cite{Ch} and Randal-Williams in \cite{RW} on rational homological stability for $S_{1^j}(M)$, we use the integral results of Segal; in Proposition A.1 of \cite{Se} he proved that $t:S_{1^j}(M) \m S_{1^{j+1}}(M)$ induces an isomorphism on integral homology for $* \leq j/2$. Therefore, we have that $t_*:H_*(W_{1^j \lambda}(M);\bZ) \m H_*(W_{1^{j+1}}(M);\bZ)$ is an isomorphism for $* \leq  min(k+j,2(k-r)+j/2-1) -1$ if $\dim M = 2$.

When we consider closed manifolds in Section \ref{secpuncturing}, the use of rational coefficients will also be unavoidable. In fact, from the presentation of the spherical braid group given in \cite{FV}, one sees that $H_1(W_{1^j2}(\C P^1);\bZ)=\bZ/(2j+2)\bZ$. Hence integral homological stability fails for closed manifolds, even in dimension two.

We also note that these techniques show that the spaces $S_{1^j \lambda}(M)$, $D_{1^j \lambda}(M)$, $W_{1^j \lambda}(M)$, $Sym_{j+k}(M)$ have stability for appropriately shifted integral compactly supported cohomology provided $M$ is a connected manifold that is the interior of a manifold with non-empty boundary and has dimension at least 2. For $\mr{Sym}_{j+k}(M)$, this is Proposition A.2 of \cite{Se}.


\end{remark}

\section{The proof for open oriented manifolds of odd dimension} \label{oddSec} In this section we explain the modifications that are necessary to the argument of the previous section if $M$ is a connected oriented manifold of odd dimension $d = 2n+1$ that is the interior of a manifold with non-empty boundary. We will need to assume that $d \geq 3$. The arguments of the previous section go through if one modifies all the statements to include the correct signs. Our goal of this section is to prove the following proposition.

\begin{proposition}\label{propoddopen} Let $M$ be a connected oriented manifold of odd dimension $d = 2n+1 \geq 3$ that is the interior of a manifold with boundary. The stabilization map $t_*: H_i(W_{1^j \lambda}(M);\bQ) \to H_i(W_{1^{j+1} \lambda}(M);\bQ)$ induces an isomorphism for $i \leq f^\mr{or}_{M,\lambda}(j)$, a function given in Equation \ref{eqnfmlambda}.\end{proposition}

The main concern involves the signs in the \p duality isomorphisms. Let $\epsilon$ denote the rational sign representation of $\fS_{k+j}$, i.e. $\bQ$ in degree zero with $\sigma \in \fS_{k+j}$ acting by multiplication with $(-1)^{\mr{sign}(\sigma)}$. Let $\widetilde{\mr{Sym}}_{k}(M)$ denote the ordered symmetric product, i.e. the product $M^{k}$, and let $\tilde{W}_{\lambda}(M)$ denote the ordered complement of the discriminant, i.e. the inverse image of $W_{\lambda}(M)$ under the quotient map $\widetilde{\mr{Sym}}_{k}(M) \to \mr{Sym}_{k}(M)$. Likewise define $\tilde{D}_{\lambda}(M)$ and $\tilde{S}_{\lambda}(M)$. One can define stabilization maps for these ordered spaces in a manner similar to the unordered case. We use the convention that the newly added point is considered the first point. 

In odd dimensions, \p duality is not equivariant. The following lemma describes a correction to the symmetric group action on compactly supported cohomology making \p duality equivariant. It will be used in the odd dimensional analogues of Lemma \ref{lemmaopensymstab} and Proposition \ref{propevenopen}.

\begin{lemma}\label{lemoddpoincarew} If $M$ is oriented of odd dimension, then one can make choices of orientations of $\tilde{W}_{1^j \lambda}(M)$ and $\widetilde{\mr{Sym}}_{k+j}(M)$ such that Poincar\'e duality gives isomorphisms 
\[(H^q_c(\tilde{W}_{1^j \lambda}(M);\bQ) \otimes \epsilon)^{\fS_{k+j}} \cong H_{2n(k+j)-q}(W_{1^j \lambda}(M);\bQ)\]
\[(H^q_c(\widetilde{\mr{Sym}}_{k+j}(M);\bQ) \otimes \epsilon)^{\fS_{k+j}} \cong H_{2n(k+j)-q}(\mr{Sym}_{k+j}(M);\bQ)\]
which are natural with respect to the inclusion and stabilization maps.\end{lemma}

\begin{proof}We have that $\widetilde{\mr{Sym}}_{k+j}(M) \cong M^{k+j}$, which has  dimension $2n(k+j)$ and is oriented by lexicographic ordering. Since $\tilde{W}_{1^j \lambda}(M)$ is an open subspace of $\widetilde{\mr{Sym}}_{k+j}(M)$ it is a manifold of the same dimension and inherits the orientation. This means that $\widetilde{\mr{Sym}}_{k+j}(M)$ and $\tilde{W}_{1^j \lambda}(M)$ have compatible Poincar\'e duality isomorphisms. We recall one can compute the rational homology or rational compactly supported cohomology of quotients of finite group actions by computing coinvariants and invariants respectively. Since the group action does not preserve the orientation but acts by the sign representation, we needed to tensor compactly supported cohomology by the sign representation to make \p duality equivariant. \end{proof}

Lemma \ref{lemoddpoincarew} means that for the application of the five lemma as in the proof of Proposition \ref{propevenopen}, we must also use $\epsilon$ in our spectral sequence computation of the compactly supported cohomology of the discriminant. The odd dimensional replacement for the spectral sequence in Lemma \ref{specEven} is then:

\begin{lemma}\label{specOdd} There is a spectral sequence converging to $(H^{p+q}_c(\tilde{D}_{1^j \lambda}(M);\bQ) \otimes \epsilon)^{\fS_{k+j}}$ with $E^1$-page 
\[E^1_{p,q} = \begin{cases} \bigoplus_{\lambda' \in \mr{col}_{p}(1^j \lambda)} (H^{p+q}_c(\tilde{S}_{\lambda'}(M);\bQ)\otimes \epsilon)^{\fS_{k+j}} & \text{if $p\geq 0$} \\
0 & \text{if $p<0$} \end{cases}\]\end{lemma}

\begin{proof}This is a consequence of invariants by a finite group action being an exact functor in characteristic zero.\end{proof}

Recall that a partition of a set $S$ is defined as a set of disjoint subsets whose union is all of $S$.  This is not to be confused with a partition of an integer. These two concepts are related as a partition of a finite has an associated partition of the cardinality of the set by remembering the cardinalities of the subsets. 

We now explain how to compute the $E^1$-page of this spectral sequence. Suppose that $\lambda$ is a partition of $k$ into $r$ integers and fix a partition $\lambda' \in \mr{col}_p(1^j \lambda)$, which is a partition of $k+j$ into $r+j-p$ integers. Let $[k+j]$ be the set $\{1,2,\ldots,k+j\}$ and $\mr{ord}(\lambda')$ be the collection of all partitions of the set $[k+j]$ that have associated partition $\lambda'$. For example, $\mr{ord}(1+2)$ consists of the three elements $\{1\}\{2,3\}$, $\{2\}\{1,3\}$ and $\{3\}\{1,2\}$. The set $\mr{ord}(\lambda')$ has a natural action of $\fS_{k+j}$. The components of $\tilde{S}_{\lambda'}(M)$ are in bijection with $\mr{ord}(\lambda')$, compatibly with the action of $\fS_{k+j}$. We denote the component corresponding to $\Lambda \in \mr{ord}(\lambda')$ by $\tilde{S}_{\Lambda}(M)$. We first state \p duality for $\tilde{S}_{\lambda'}(M)$ with the correct signs to make it equivariant:

\begin{lemma}One can make choices of orientations of $\tilde{S}_{\lambda'}(M)$ such that Poincar\'e duality gives an isomorphism \[H^q_c(\tilde{S}_{\lambda'}(M);\bQ) \cong \bigoplus_{\Lambda \in \mr{ord}(\lambda')} H_{(2n+1)(r+j-p)-q}(\tilde{S}_{\Lambda}(M);\bQ)\]
which is $\fS_{k+j}$-equivariant, with the $\fS_{k+j}$ action on the left-hand side induced by the action on $\widetilde{ \mr{Sym}}_{k+j}(M)$ and the action on the right-hand side the induced action multiplied by a sign which we describe in the proof.\end{lemma}

\begin{proof}We claim that the components of $\tilde{S}_{\lambda'}(M)$ are orientable manifolds. For simplicity, we assume that $M$ is smooth. However, this assumption is not essential. We will pick an orientation of the tangent bundle for each component, denoted $\tilde{S}_{\Lambda}(M)$ for some partition $\Lambda$ of the set $[k+j]$ with associated partition $\lambda'$. An orientation of the tangent bundle of $\tilde{S}_{\Lambda}(M)$ is given as follows. Fix $\Lambda \in \mr{ord}(\lambda')$ and let $\varpi_1,\ldots,\varpi_{r+j-p}$ denote the terms in the partition, in lexicographic order by first decreasing size and then increasing smallest element. For example, if $\Lambda = \{1\}\{4\}\{2,3\}$, then $\varpi_1 = \{2,3\}$, $\varpi_1 = \{1\}$ and $\varpi_3 = \{4\}$. Any element $x \in \tilde{S}_{\Lambda}(M)$ can be described by distinct points $x_1,\ldots,x_{r+j-p}$ labeled by $\varpi_1,\ldots,\varpi_{r+j-p}$ respectively, so that the ordering gives an isomorphism $\tilde{S}_\Lambda(M) \cong \tilde{S}_{1^{r+j-p}}(M)$. Ordered configuration spaces of an oriented manifold have a canonical orientation by taking the orientations at each of the points in the configuration in order.

The ordering of the terms in each partition gives a bijection $\Lambda \cong [r+j-p]$.  By using the explicit orientations given above we note that if $\sigma \in \fS_{k+j}$ translates from the $\Lambda$-component to the $\Lambda'$-component it multiplies the orientation by the $n$th power of the sign of $\sigma$ considered as an element of $\fS_{r+j-p}$ via $[r+j-p] \cong \Lambda \to \Lambda' \cong [r+j-p]$. This is the extra sign described in the statement of the proof and we denote this by $\epsilon(\Lambda \to \Lambda')$.
\end{proof}

Since we are interested in proving stability for $(H^*_c(\tilde{D}_{1^j \lambda}(M);\bQ) \otimes \epsilon)^{\fS_{k+j}}$, we are not interested in the invariants of the compactly supported cohomology groups of the strata. Instead the relevant groups are the invariants of its tensor product with the sign representation.

Fix for each $\lambda' \in \mr{col}_p(1^j\lambda)$ a choice $\Lambda \in \mr{ord}(\lambda')$ and let $\mr{Stab}(\Lambda) \subset \fS_{k+j}$ be the permutations preserving the partition $\Lambda$. Suppose that the partition $\lambda'$ has $n(1)$ terms of length $1$, $n(2)$ terms of length $2$, etc. There is a natural isomorphism $\mr{Stab}(\Lambda) \cong \prod_{l=1}^{\infty} \fS_l \wr \fS_{n(l)}$, a product of wreath products of symmetric groups. Let $\eta_{\Lambda}$ be the representation of this group given by the composition of projection to $\prod_{l=2}^{\infty} \fS_l \wr \fS_{n(l)}$ and taking the product of the signs of its canonical homomorphisms to $\fS_{\sum_{l=2}^\infty ln(l)}$ (given by block sum) and $\prod_{l=2}^\infty \fS_{n(l)}$ (given by projection). Then define $\cH_*(S_{\lambda'}(M);\bQ))$ to be the $\mr{Stab}(\Lambda)$-coinvariants of $H_*(\tilde{S}_{\Lambda}(M);\bQ) \otimes \eta_{\Lambda}$.

\begin{lemma}\p duality gives an isomorphism \[(H^q_c(\tilde{S}_{\lambda'}(M);\bQ))\otimes \epsilon)^{\fS_{k+j}} \cong \cH_{(2n+1)(r+j-p)-q}(S_{\lambda'}(M);\bQ)\]
\end{lemma}

\begin{proof}The invariants in the statement of the lemma are given by the coinvariants of direct sum $\bigoplus_{\Lambda \in \mr{ord}(\lambda')} H_{*}(\tilde{S}_{\Lambda}(\bR^n);\bQ)$ by $\fS_{k+j}$ where $\fS_{k+j}$ acts by the induced maps and a sign given by the product of $\epsilon$ and $\epsilon(\Lambda \to \Lambda')$ if translating from $\Lambda$ to $\Lambda'$. Since the action on the components is transitive, this is naturally isomorphic to the coinvariants by $\mr{Stab}(\Lambda)$ of $H_{*}(\tilde{S}_{\Lambda}(\bR^n);\bQ)$ for some $\Lambda \in \mr{ord}(\lambda')$. We can identify $\mr{Stab}(\Lambda)$ with the permutations that preserve the partition $\Lambda$. Now the action is via the induced maps and a sign is given by the product of $\epsilon$ and $\epsilon_{j+r-p}$, the latter being given by the homomorphism $\mr{Stab}(\Lambda) \to \fS_{j+r-p}$ by the action of the terms of the partition $\Lambda$, ordered lexicographically as before to be identified with the ordered set $[j+r-p]$. In particular, the signs coming from the length $1$ terms cancel against the sign $\epsilon$. The other signs combine to form $\eta_\Lambda$.\end{proof}

For even dimensional manifolds, we used the result that the spaces $S_{1^j\lambda}(M)$ have homological stability. In odd dimensions, we will instead need to know that the groups $\cH_*(S_{1^j\lambda'}(M);\bQ))$ stabilize. 

\begin{lemma}\label{lemopenstratastabodd} Let $\lambda$ be a partition with $i$ $1$'s. The map $t_*: \cH_*(S_{\lambda}(M);\bQ) \to \cH_*(S_{1\,\lambda}(M);\bQ)$ is an isomorphism in the following ranges: (i) $* \leq i$ if $M$ is of dimension $d>2$, (ii) $* < i$ if $M$ is of dimension $d=2$, and (iii) $* < (a+1)i$ if condition $(*)_a$ holds.\end{lemma}

\begin{proof}Recall that $\cH_*(S_{\lambda}(M);\bQ)$ and $\cH_*(S_{1\, \lambda}(M);\bQ)$ were defined as coinvariants of some group actions on $H_*(\tilde{S}_{\Lambda}(M);\bQ)$ and $H_*(\tilde{S}_{1\, \Lambda}(M);\bQ)$ respectively with $\Lambda \in \mr{ord}(\lambda)$.

Note that $\tilde{S}_{\Lambda}(M)$ is a configuration with $i$ particles labeled by sets of cardinality $1$ and the remaining particles labeled by $r$ bigger sets. Let $\Lambda'$ consist of those subsets in $\Lambda$ that do not have cardinality $1$. Then there is a fibration
\[\tilde{S}_{1^{i}}(M \backslash \{r \text{ points}\}) \to \tilde S_{\Lambda}(M) \to \tilde{S}_{\Lambda'}(M)\] This fibration is $\mr{Stab}(\Lambda) \cong \prod_{l=1}^\infty \fS_l \wr \fS_{n(l)}$-equivariant where the action on the first term is via the projection to $\fS_1 \wr \fS_{n(1)} = \fS_{i}$ (note that $n(1)=i$), on the middle term is via the original action and on the last term is via the projection to $\prod_{l=2}^\infty \fS_l \wr \fS_{n(l)}$. Going through the construction of the Serre spectral sequence, but now including the sign coming from $\eta_{1^{j}\Lambda}$ and taking coinvariants (which is an exact functor with rational coefficients) we get a spectral sequence with 
\[E^2_{p,q} = \left( H_p\left(\tilde{S}_{\Lambda'}(M),H_q(\tilde{S}_{1^i}(M \backslash \{r\text{ points}\}))\right) \otimes \eta_{1^{j}\Lambda}\right)_{\mr{Stab}(\Lambda)} \Rightarrow \cH_{p+q}(\tilde{S}_{1^{j} \lambda}(M))\] Since the term $\fS_1 \wr \fS_{n(1)} = \fS_{i}$ acts trivially on the base, we can identify the $E^2$-page with 
\[\left( H_p\left(\tilde{S}_{\Lambda'}(M),H_q(S_{1^{i}}(M \backslash \{r\text{ points}\}))\right) \otimes \eta_{1^{j}\Lambda}\right)_{\prod_{l=2}^\infty \fS_l \wr \fS_{n(l))}}\] Considering the same construction for $1 \Lambda$, we see that we get a map of spectral sequence inducing the stabilization map on the $E^\infty$-pages. On the $E^2$-page this map is induced by the stabilization map of the fibers 
\[S_{1^{i}}(M \backslash r\text{ points}) \to S_{1^{i+1}}(M \backslash r \text{ points}))\] which is a stabilization map between unordered configuration spaces. This is known to be an isomorphism in the ranges in the statement of the lemma. Since the base and the coinvariants of the group action are the same, we get an isomorphism on the $E^2$-page in a range. The lemma now follows from a spectral sequence comparison theorem.
\end{proof}

Given these modified lemmas as input, the proof of Proposition \ref{propoddopen} now follows along the lines of the proof of Proposition \ref{propevenopen}.

\section{The proof for non-orientable manifolds} \label{secnonorientable} In this section we explain how to modify the proofs of the previous two sections to non-orientable manifolds. This is similar to how Segal treats the non-orientable case in Appendix A of \cite{Se}.

\begin{proposition}\label{propallguys} Let $M$ be a connected non-orientable manifold $M$ of dimension $\geq 2$ that is the interior of a manifold with boundary. The stabilization map $t_*: H_i(W_{1^j \lambda}(M);\bQ) \to H_i(W_{1^{j+1} \lambda}(M);\bQ)$ induces an isomorphism for $i \leq f^\mr{nor}_{M,\lambda}(j)$, with
\begin{equation}\label{eqnfmnonlambda} f^\mr{nor}_{M,\lambda}(j) = \begin{cases} \min(k+j,d(k-r)+j-1)-1 & \text{if $\dim M = d >2$} \\
\min(k+j,2(k-r)+j/2-1)-1 & \text{if $\dim M = 2$}\end{cases}
\end{equation}\end{proposition}

\begin{proof}Again the even dimensional case is easier, so we start with that. In this case, instead of considering the compactly supported cohomology of $\mr{Sym}_{k+j}(M)$, $D_{1^j \lambda}(M)$ and $W_{1^j \lambda}(M)$ with coefficients in $\bQ$, one considers compactly supported cohomology with coefficients in the rational orientation local coefficient system $\omega$. This is the local coefficient system given by taking the rational orientation local coefficient system on $M^{k+j}$ and noting that if the dimension is even, this descends to a local coefficient system on $\mr{Sym}_{k+j}(M)$. We obtain local coefficient systems on  $D_{1^j \lambda}(M)$ and $W_{1^j \lambda}(M)$ by pullback. We denote all of these local systems by $\omega$ as well.

The reason for considering this local coefficient system is that \p duality gives isomorphisms
\begin{align*}H_*(\mr{Sym}_{k+j}(M);\bQ) &\cong H^{2n(k+j)-*}_c(\mr{Sym}_{k+j}(M);\omega)\\
H_*(W_{1^j \lambda}(M);\bQ) &\cong H^{2n(k+j)-*}_c(W_{1^j \lambda}(M);\omega)\end{align*} The proofs given in Section \ref{evenSec} then go through in essentially the same manner, except we need to be more careful in Lemma \ref{lemopenstratastab} and Corollary \ref{coropenstratastab}. First of all, in the non-orientable case Church's results do not apply and if $\dim(M) = 2$ we need to use the range in Proposition A.1 of \cite{Se} (or alternatively Theorem A of \cite{RW}), which hold rationally by the universal coefficient theorem. Secondly, though Lemma \ref{lemopenstratastab} is true for non-orientable manifolds it is not relevant for proving stability for the groups $H^*_c(D_{1^j \lambda}(M);\omega)$, appropriately shifted. To be useful for proving that $H^*_c(D_{1^j \lambda}(M);\omega)$ stabilizes, our replacement for Corollary \ref{coropenstratastab} should involve stability for $H^*_c(S_{1^j \lambda}(M);\omega)$. Unfortunately,  $\omega$ is \textit{not} equal to the orientation local coefficient system $\omega'$ of the stratum $S_{1^j \lambda}(M)$. Thus, the desired analogue of Lemma \ref{lemopenstratastab} concerns compactly supported cohomology with coefficients in the tensor products of $\omega$ and $\omega'$. The coefficient system $\omega \otimes \omega'$ is trivial on the fiber of the fiber bundle \[S_{1^{i}}(M \backslash \{r \text{ points}\}) \to S_{\lambda}(M) \to S_{\lambda'}(M)\] Therefore, we can still use spectral sequence comparison and untwisted homological stability for $S_j(M \backslash \{r \text{ points}\})$ to conclude that $H_*(S_{1^j\lambda}(M);\omega \otimes \omega')$ stabilizes. \p duality now gives stability for $H^*_c(S_{1^j\lambda}(M);\omega)$. The rest of the modifications are straightforward. 

In the odd dimensional case, one similarly modifies the proof by using the orientation local coefficient system on the ordered version of the spaces, $\widetilde{Sym}_{k+j}(M)$, $\tilde{W}_{1^j \lambda}(M)$, $\tilde{D}_{1^j \lambda}(M)$ and $\tilde{S}_{\lambda'}(M)$. This does not substantially change the arguments of Section \ref{oddSec} and one does not need to introduce any additional signs to make \p duality equivariant. The modifications to Lemma \ref{lemopenstratastabodd} are similar to the modifications to Lemma \ref{lemopenstratastab} and Corollary \ref{coropenstratastab} described in the previous paragraph.\end{proof}

\section{The proof for closed manifolds by puncturing}\label{secpuncturing} In this section, we prove homological stability for the spaces $W_{1^j \lambda}(M)$ for $M$ closed. One cannot define stabilization maps for closed manifolds as there is no way to add an extra point. Instead we use the transfer map. Our proof follows similar arguments to those used by Randal-Williams in Section 9 of \cite{RW} to leverage homological stability for configuration spaces of particles in open manifolds to prove homological stability for configuration spaces of particles in closed manifolds. We will first recall the definition of the transfer map $ \tau: H_*(W_{1^{j+1} \lambda}(M);\Q) \to H_*(W_{1^{j} \lambda}(M);\Q)$. We then prove that the transfer map induces an isomorphism in the same range as the stabilization map when the manifold is open. Using an augmented simplicial space, we describe a spectral sequence computing $H_*(W_{1^{j} \lambda}(M);\bQ) $ in terms of $H_*(W_{1^{j} \lambda}(N);\bQ) $ with $N$ equal to $M$ minus a finite number of points. The transfer map will respect this spectral sequence. The theorem will follow by comparing the spectral sequence for $W_{1^{j} \lambda}(M)$ with the one for $W_{1^{j+1} \lambda}(M)$.

Let $\lambda$ be a partition of $k$ and recall that $\tilde{W}_{1^j \lambda}(M)$ is the ordered version of $W_{1^j \lambda}(M)$. For $i \leq j$, let $\mr{del}_{i,j} : \tilde W_{1^j \lambda}(M) \m \tilde W_{1^i \lambda}(M)$ be the map which deletes the last $j-i$ points of $\tilde W_{1^j \lambda}(M)$. This makes sense since the points of $\tilde W_{1^j \lambda}(M)$ are ordered. The map $del_{i,j}$ is $\fS_{i+k}$-equivariant. Here $\fS_{i+k}$ acts on $ \tilde W_{1^j \lambda}(M) $ via the inclusion of $\fS_{i+k}$ into $\fS_{j+k}$ induced by the standard inclusion $\fS_i$ into $\fS_j$. Thus it induces a map $(\mr{del}_{i,j})_*:H_*(\tilde W_{1^{j} \lambda}(M);\bQ)_{\fS_{i+k}} \to H_*(\tilde W_{1^{i} \lambda}(M);\bQ)_{\fS_{i+k}}$. Since we are viewing  $\fS_{i+k}$ as a subgroup of  $\fS_{j+k}$, we get an inclusion $\iota: H_*(\tilde W_{1^{j} \lambda}(M);\bQ)_{\fS_{j+k}} \to H_*(\tilde W_{1^{j} \lambda}(M);\bQ)_{\fS_{i+k}}$. Using this we can define the transfer map. 

\begin{definition}
The transfer map $\tau_{i,j}: H_*(W_{1^{j} \lambda}(M);\bQ) \to H_*(W_{1^{i} \lambda}(M);\bQ)$ is defined as \[ (\mr{del}_{i,j})_* \circ  \iota: H_*(\tilde W_{1^{j} \lambda}(M);\bQ)_{\fS_{j+k}} \m H_*(\tilde W_{1^{i} \lambda}(M);\bQ)_{\fS_{i+k}} \] postcomposed and precomposed with the natural isomorphisms $H_*(W_{1^{i} \lambda}(M);\bQ) \cong H_*(\tilde W_{1^{i} \lambda}(M);\bQ)_{\fS_{j+k}} $ and $H_*(W_{1^{j} \lambda}(M);\bQ) \cong H_*(\tilde W_{1^{j} \lambda}(M);\bQ)_{\fS_{j+k}} $ respectively.

\end{definition}

For $i = j-1$, we denote $\tau_{i,j}$ by $\tau$.  We now show that $\tau$ induces a rational homology equivalence in a range for $M$ open by proving it is the inverse to the stabilization map in the stable range. 

\begin{lemma}\label{lemtransferinverse} Let $M$ be the interior of a connected manifold with boundary. The transfer map $\tau: H_*(W_{1^{j+1} \lambda}(M);\bQ) \to H_*(W_{1^{j} \lambda}(M);\bQ)$ is an inverse to the stabilization map in the range where the stabilization map is an isomorphism.
\end{lemma}

\begin{proof}Suppose that $\lambda$ is a partition of $k$, then for $-k \leq j \leq 0$ we let $W_{1^j \lambda}(M)$ be the inverse image of $W_{\lambda}(M)$ in $\mr{Sym}_{k+j}(M)$ under $t^{-j}$. Fix $i \geq 0$, then we set $B_j = H_i(W_{1^{j-k} \lambda}(M);\bQ)$. Then we define $\sigma_j: B_{j-1} \to B_j$ to be the map on homology induced by the stabilizing map for $j \geq 1$ and $\sigma_0$ to be $0 \to B_0$. The transfer gives maps $\tau_{q,p}: B_p \to B_q$. These satisfy $\tau_{q,p} \circ \sigma_p = \tau_{q,p-1} + \sigma_{p-1} \circ \tau_{q-1,p-1}$ and $\tau_{p,p} = \mr{id}$. By Lemma 2.2 of \cite{Do}, the map 
\[\oplus_{q \leq p} \tau_{q,p}: B_p \to \bigoplus_{0 \leq q \leq p} B_q/\mr{im}(\sigma_q)\]
is an isomorphism. Using this one can conclude that $\tau_{p-1,p} \circ \sigma_p$ respects this decomposition and on each summand is multiplication by a non-zero scalar. More precisely, first note that we have that $\tau_{m,m+1} \circ \ldots \circ \tau_{p-1,p} = (p-m)! \tau_{m,p}$ and thus ${p-q \choose p-m} \tau_{q,p} = \tau_{q,m} \circ \tau_{m,p}$ rationally. We now compute the composition of $\tau_{p-1,p} \circ \sigma_p$ with one of the maps in the decomposition. Let $\pi_q$ be the projection $B_q \to B_q/\mr{im}(\sigma_q)$. We have that the projection to the $q$-summand of $ B_{p-1}$ is given by the map $\pi_q \circ \tau_{q,p}$. Now note that 
\begin{align*}\pi_q \circ \tau_{q,p-1} \circ \tau_{p-1,p} \circ \sigma_p &= (p-q) \pi_q \circ \tau_{q,p} \circ \sigma_p \\
&= (p-q) \pi_q \circ (\tau_{q,p-1} + \sigma_q \circ \tau_{q-1,p-1}) \\
&= (p-q) \pi_q \circ \tau_{q,p-1}\end{align*}
But $\pi_q \circ \tau_{q,p-1}$ is exactly the projection onto the $q$-summand of $B_{p-1}$. Since we are working rationally, this implies that $\tau_{p-1,p}$ is an isomorphism when $\sigma_p$ is. Specializing to $p = k+j$ gives the desired result.
\end{proof}

Recall that a semisimplicial object is defined in the same way as a simplicial object without the data of degeneracy maps. We now describe a semisimplicial space $\tilde{\cW}_\bullet(\lambda)$ with augmentation to $\tilde{W}_{1^j \lambda}(M)$. 

\begin{definition}
The space of $p$ simplices of  $\tilde{\cW}_\bullet(\lambda)$ is given by: 
\[\tilde{ \cW_p}(\lambda) = \bigsqcup_{\{m_0,\ldots,m_p\} \in {S_{1^{p+1}(M)}}} \tilde{W}_{\lambda} (M \backslash \{m_0,\ldots,m_p\})\] The $i$th face map is induced by the inclusion $M \backslash \{m_0,\ldots,m_p\} \m M \backslash \{m_0,\ldots,m_{p-1}\}$ which fills in the $i$th puncture. 
\end{definition}

The above construction works equally well for $p=-1$, giving it the structure of an augmented semisimplicial space. Note that $\tilde \cW_{-1}(\lambda) =\tilde W_{\lambda}(M)$. We will show that the augmentation map induces a weak equivalence $||\tilde{\cW}_\bullet(\lambda)|| \to \tilde{W}_{ \lambda}(M)$. For this reason we will call $\tilde{\cW}_\bullet(\lambda)$ a resolution of $\tilde{ \cW_p}(\lambda)$. This resolution is useful as $M \backslash \{m_0,\ldots,m_p\}$ is an open manifold and so we will be able to apply Lemma \ref{lemtransferinverse} levelwise to a semisimplicial chain complex constructed from $\tilde{\cW}_\bullet(\lambda)$. 

To prove that the augmentation is a weak equivalence, we will first recall the definitions of microfibrations and flag sets. We are interested in these definitions since every microfibration with weakly contractible fibers is a weak equivalence and there is an easily checked condition for the contractiblity of the geometric realization of a flag set.

\begin{definition}A map $f: E \to B$ is called a microfibration if for $m \geq 0$ and each commutative diagram 
\[\xymatrix{\{0\} \times D^m \ar[r] \ar[d] & E \ar[d] \\
[0,1] \times D^m \ar[r] &B}\]
there exists an $\epsilon \in (0,1]$ and a partial lift $[0,\epsilon] \times D^m \to E$ making the resulting diagram commute.\end{definition}

The following proposition was proven by Weiss in Lemma 2.2 of \cite{weissclassify}.

\begin{proposition}
 A microfibration with weakly contractible fibers is a weak equivalence.
\end{proposition}

We now define flag sets, a type of simplicial set where $p$-simplices are defined by their vertices. 


\begin{definition}
A semisimplicial set $X_\bullet$ is said to be a flag set if the $p$-simplices $X_p$ are a subset of $X_0^{p+1}$ and if an ordered $(p+1)$-tuple $(v_0,\ldots,v_p)$ forms a $p$-simplex if and only if $(v_i,v_j)$ forms a $1$-simplex for all $i \neq j$.
\end{definition}

\begin{lemma}\label{lemflagcontr} Let $X_\bullet$ be a flag set such that for each finite collection $\{v_1,\ldots,v_N\}$ of $0$-simplices there exists a $0$-simplex $v$ such that $(v_i,v)$ is a $1$-simplex for all $i$. Then $||X_\bullet||$ is weakly contractible.\end{lemma}

\begin{proof}Let $f: S^i \to ||X_\bullet||$ be arbitrary. By simplicial approximation, we can homotope $f$ to a map $g$ which is simplicial with respect to some PL-triangulation of $S^i$. Note that the image of $g$ is contained in a finite subsemisimplicial set $X'_\bullet$ of $X_\bullet$ spanned by some set of $0$-simplices $\{v_1,\ldots,v_N\}$. By hypothesis the join $||X'_\bullet|| * \{v\}$ is a subcomplex of $||X_\bullet||$ and thus we can extend the map $g$ to a map $g:\mr{Cone}(S^i) \to ||X_\bullet||$ by sending the cone point to $v$.\end{proof}

We now prove that the augmentation is a weak equivalence. 

\begin{proposition}
The augmentation induces a weak equivalence $||\tilde{\cW}_\bullet(\lambda)|| \to \tilde{W}_{ \lambda}(M)$.
\end{proposition}

\begin{proof}

We will prove that this map is a microfibration with contractible fibers. To see that it is a microfibration, suppose we have a map $f: D^n \times [0,1] \to \tilde{W}_{1^j \lambda}(M)$ and a lift $\hat{f}$ to $||\tilde{\cW}_\bullet(j)||$ on defined $D^n \times \{0\}$. Given a point $y \in \tilde{W}_{1^j \lambda}(M)$, the extra data needed to lift $y$ to $||\tilde{\cW}_\bullet(j)||$ is a simplicial coordinate $t \in \mr{int}(\Delta^p)$ and a configuration $(m_0,\ldots,m_p) \in \tilde S_{1^{p+1}}(M)$ such that $m_0,\ldots,m_p$ are disjoint from the points of $y'$. Note that if $y'$ is sufficiently close to $y$, the points $m_0,\ldots,m_p$ will also be disjoint from $y$. Therefore, the data used to lift $y$ will also define a lift of $y'$. Since $f$ is continuous, for any $x \in D^n$, using the configuration and simplicial coordinate associated to $\hat f(x)$, we can define a lift on of $f$ on $\{0\} \times [0,\epsilon_x)$ for some $\epsilon_x >0$. Since $\hat f$ is also continuous, we can define a lift on $U_x \times [0,\epsilon_x/2)$ with $U_x$  a neighborhood of $x$ in $D^n$. This lift is again defined using the simplicial coordinate and configuration associated to $\hat f(x')$ to lift the point $f(x',s)$ to a point in $||\tilde{\cW}_\bullet(j)||$. By compactness of $D^n$, we can find one choice of $\epsilon$ such that this construction defines a lift on all of $ D^n \times [0,\epsilon)$.

Next we will show that the fibers of the augmentation map are contractible. Note that the fiber of the augmentation map over a configuration $y \in \tilde{W}_{ \lambda}(M)$ is homeomorphic to the geometric realization of the following flag set which we denote $F_{\bullet}(y)$. The set of $p$-simplices of $F_{\bullet}(y)$ is the underlying set of $\tilde S_{1^{p+1}}(M \backslash y)$ and the face maps are induced by forgetting the $i$th point. It is clear that $F_{\bullet}(y)$ is a flag set. It satisfies the conditions of Lemma \ref{lemflagcontr} since we can always find a point in $M$ not contained in some fixed finite subset. Therefore the fibers are weakly contractible and so the augmentation is a weak equivalence.
\end{proof}

We now prove that the transfer map induces a homology equivalence in a range for arbitrary manifolds of dimension at least $2$.

\begin{theorem}Let $M$ be any connected manifold of dimension at least $2$. The transfer map $\tau: H_i(W_{1^{j+1} \lambda}(M);\bQ) \to H_i(W_{1^{j} \lambda}(M);\bQ)$ is an isomorphism for $i \leq f^\mr{or}_{M,\lambda}(i)$ in the orientable case or $i \leq f^\mr{nor}_{M,\lambda}(i)$ in the non-orientable case, the functions given in Equation \ref{eqnfmlambda} and \ref{eqnfmnonlambda}.\end{theorem}

\begin{proof}First we will extend the transfer map to the symmetric group invariant singular chains of the resolution. Applying rational singular chains and using that geometric realisation commutes with singular chains up to quasi-isomorphism gives us a semisimplicial chain complex $C_*(\tilde{\cW}_\bullet(j);\bQ)$ such that the augmentation $||C_*(\tilde{\cW}_\bullet(1^j \lambda);\bQ)|| \to C_*(\tilde{W}_{1^j \lambda}(M);\bQ)$ is a quasi-isomorphism. Applying $\fS_{k+j}$-coinvariants levelwise we get a semisimplicial complex with augmentation to a chain complex with homology $H_*(W_{1^j \lambda}(M);\bQ)$ and level $p$ having homology given by
\[\bigoplus_{\{m_0,\ldots,m_p\} \in C_{p+1}(M)} H_*(W_{1^j \lambda}(M \backslash \{m_0,\ldots,m_p\});\bQ)\]

\noindent  Applying the construction of the transfer map levelwise to the augmented semisimplicial chain complex gives us a semisimplicial chain map $(\tau_\bullet)_*: C_*(\cW_\bullet(j+1);\bQ) \to C_*(\cW_\bullet(j);\bQ)$ inducing the transfer map on homology levelwise. Recall that there is a spectral sequence converging to the homology of a geometric realization of a semisimplicial chain complex in terms of the homology of the levels. The transfer map induces a map between the spectral sequence for $1^j \lambda$ and $1^{j+1} \lambda$. Since $M \backslash \{m_0,\ldots,m_p\}$ are connected manifolds that are the interior of a manifold with boundary, the previous lemma implies that $(\tau_\bullet)_*$ induces an isomorphism on $E^1_{p,q}$ for $q \leq f^\mr{or}_{M,\lambda}(j)$ or $q \leq f^\mr{nor}_{M,\lambda}(j)$. By a spectral sequence comparison theorem, the transfer map $\tau: H_*(W_{1^{j+1} \lambda}(M);\bQ) \to H_*(W_{1^j \lambda}(M);\bQ)$ on the augmentations is an isomorphism in the range given in the statement of the theorem.\end{proof}

\bibliographystyle{amsalpha}
\bibliography{thesis4}

\end{document}